\documentclass{article}%
\usepackage{amsmath}
\usepackage{amsfonts}
\usepackage{amssymb}
\usepackage{graphicx}
\usepackage{color}%
\providecommand{\U}[1]{\protect\rule{.1in}{.1in}}

\newtheorem{theorem}{Theorem}

\newtheorem{claim}[theorem]{Claim}

\newtheorem{definition}[theorem]{Definition}

\newtheorem{lemma}[theorem]{Lemma}

\newtheorem{proposition}[theorem]{Proposition}
\newtheorem{remark}[theorem]{Remark}

\newenvironment{proof}[1][Proof]{\noindent\textbf{#1} }{\ \rule{0.5em}{0.5em}}

\numberwithin{equation}{section}
\numberwithin{theorem}{section}

\begin{document}
\author{}
\date{}
\title{\scshape Exterior convexity and classical calculus of variations }
\maketitle

\centerline{\scshape Saugata Bandyopadhyay }
\medskip
{\footnotesize
\centerline{Department of Mathematics \& Statistics}
 \centerline{IISER Kolkata}
   \centerline{Mohanpur-741246, India}
   \centerline{saugata.bandyopadhyay@iiserkol.ac.in}

}

\medskip

\centerline{\scshape Swarnendu Sil}
\medskip
{\footnotesize
 \centerline{ Section de Math\'{e}matiques}
   \centerline{Station 8, EPFL}
   \centerline{1015 Lausanne, Switzerland}
   \centerline{swarnendu.sil@epfl.ch}
}

\begin{abstract}
We study the relation between various notions of exterior convexity introduced in Bandyopadhyay-Dacorogna-Sil \cite{BDS1} with the classical notions of rank one convexity, quasiconvexity and polyconvexity. To this end, we introduce a projection map, which generalizes the alternating projection for two-tensors in a new way and study the algebraic properties of this map.
We conclude with a few simple consequences of this relation which yields new proofs for some of the results discussed in Bandyopadhyay-Dacorogna-Sil \cite{BDS1}.
\end{abstract}

\section{Introduction}
The notion of exterior convexity introduced in Bandyopadhyay-Dacorogna-Sil \cite{BDS1} is of fundamental importance in calculus of variations on exterior spaces, playing a role similar to what is played by the usual notions of convexity in classical vectorial calculus of variations.

\par However, the precise connection between these two sets of notions of convexity is  a question of somewhat delicate balance.
In this article, we explore this connection through the introduction of an appropriate projection map. While this projection map coincides with the canonical alternating projection of the two-tensor fields onto the exterior two-forms, it is non-trivial in the context of higher order forms. Furthermore, the projection map has the crucial property that it projects the tensor product to the exterior product and the gradient to the exterior derivative. It also allows us to express the connection between the notions of
 exterior convexity and classical notions of convexity in a crisp and explicit way, which is the content of our main theorem stated as follows
 \begin{theorem}\label{intro main thm}
Let $2\leq k\leq n$, $f:\Lambda^{k}  \rightarrow\mathbb{R}$ and $\pi:\mathbb{R}^{\tbinom{n}{k-1}\times n}\rightarrow
\Lambda^{k}$ be the projection map. Then the following equivalences hold%
\begin{align*}
f\text{ ext. one convex }&\Leftrightarrow\text{ }f\circ\pi\text{ rank one
convex. }\\
f\text{ ext. quasiconvex }&\Leftrightarrow\text{ }f\circ\pi\text{ quasiconvex. }\\
f\text{ ext. polyconvex }&\Leftrightarrow\text{ }f\circ\pi\text{ polyconvex. }
\end{align*}
\end{theorem}
\par The aforementioned result essentially situates the circle of ideas discussed in Bandyopadhyay-Dacorogna-Sil \cite{BDS1} in its proper place with respect to classical calculus of variations, which is
well-developed and by now, standard (cf. Dacorogna \cite{DCV2}). It allows us to do calculus of variations back and forth between exterior spaces and the space of matrices. In particular, some results which were directly proved in Bandyopadhya-Dacorogna-Sil \cite{BDS1} turn out be easy corollaries of the theorem aforementioned above, in conjunction with classical results of vectorial calculus of variations. Notable among them is the characterization theorem for ext. quasiaffine functions (compare the proof of Theorem 3.3 in Bandyopadhyay-Dacorogna-Sil \cite{BDS1} with that of Theorem \ref{Thm principal quasiaffine}). While in this process we do sacrifice the intrinsic character and the co-ordinate free advantage of a direct proof in exterior spaces, a simple proof is obtained nonetheless provided we are ready to assume the results of classical calculus of variations which are non-trivial and technical in their own right.

\par In this article, our main goal is to prove the aforementioned theorem. While proving the first two equivalences in Theorem \ref{intro main thm} is easy
from the definition of the projection map, proving the third one turns out to be surprizingly difficult and is of our principal concern in this article. One of the obstacles to the proof is the burden of heavy notations. To clarify presentation and to facilitate bookkeeping, we employed a system of notations, which is explained in detail in Section \ref{notations} at the end of the article. However, once the cloud of heavy notations is cleared, the proof highlights many intricacies of the
algebraic structure of alternating multilinear maps, namely the algebraic structure of determinants and minors and their interrelationship with the algebra
 of the wedge products which we believe should be of independent interest.

 \par The rest of the paper is organized as follows: In Section 2, we recall the definitions of exterior convexity and introduce the projection map.
 Section 3 states the main theorem and presents the consequences along with a characterization theorem and a weak lower semicontinuity result. Section 4 explores the
 algebraic structure of the projection map is greater detail and Section 5 is devoted to the proof of an instrumental lemma, which singles out the crux of
 the matter. We conclude the proof of the main theorem in Section 6. Finally, the notations used throughout the article is explained in Section 7.
\section{Preliminaries}
\subsection{Notions of exterior convexity}
We start by recalling the notions of exterior convexity as introduced in \cite{BDS1}.
\begin{definition}
Let $1\leq k\leq n$ and $f:\Lambda^{k}\rightarrow\mathbb{R}.$\smallskip

\textbf{(i)} We say that $f$ is \emph{ext. one convex}, if the function%
\[
g:t\rightarrow g\left( t\right)  =f\left( \xi+t\,\alpha\wedge\beta\right)
\]
is convex for every $\xi\in\Lambda^{k},$ $\alpha\in\Lambda^{k-1}$ and
$\beta\in\Lambda^{1}.$ If the function $g$ is affine we say that $f$ is
\emph{ext. one affine.}\smallskip

\textbf{(ii)} A Borel measurable and locally bounded function $f$ is said to
be \emph{ext. quasiconvex}, if the inequality%
\[
\int_{\Omega}f\left( \xi+d\omega\right)  \geq f\left( \xi\right)
\operatorname*{meas}\Omega
\]
holds for every bounded open set $\Omega \subset \mathbb{R}^{n}$, $\xi\in\Lambda^{k}$ and
$\omega\in W_{0}^{1,\infty}\left( \Omega;\Lambda^{k-1}\right)$. If
equality holds, we say that $f$ is \emph{ext. quasiaffine.}\smallskip

\textbf{(iii)} We say that $f$ is \emph{ext. polyconvex}, if there exists a
convex function%
\[
F:\Lambda^{k}\times\cdots\times\Lambda^{\left[  n/k\right]
k}\rightarrow\mathbb{R}%
\]
such that%
\[
f\left( \xi\right)  =F\left( \xi,\cdots,\xi^{\left[  n/k\right]
}\right),\text{ for all }\xi \in\Lambda^{k}.
 \]
If $F$ is affine, we say that $f$ is \emph{ext. polyaffine.}
\end{definition}
There are analogous notions of interior convexity (cf. \cite{BDS1}). In what follows, we will discuss the case of exterior convexity only.
The case of interior convexity can be derived from the case for exterior convexity by means of Hodge duality.
\subsection{Projection maps}
To study the relationship between the notions introduced
in \cite{BDS1} and the classical notions of the vectorial calculus of variations
namely rank one convexity, quasiconvexity and polyconvexity (see
\cite{DCV2}), we will introduce a projection map. We first introduce some notations. As usual, by abuse
of notations, we identify $\Lambda^{k}\left( \mathbb{R}^{n}\right)  $ with
$\mathbb{R}^{\tbinom{n}{k}}.$
\begin{definition}[Projection map]
Let $2\leq k\leq n.$ We write a matrix $\Xi\in\mathbb{R}^{\tbinom{n}{k-1}\times n},$
the upper indices being ordered alphabetically, as%
\begin{align*}
\Xi &  =\left(
\begin{array}
[c]{ccc}%
\Xi_{1}^{1\cdots\left(  k-1\right)  } & \cdots & \Xi_{n}^{1\cdots\left(
k-1\right)  }\\
\vdots & \ddots & \vdots\\
\Xi_{1}^{\left(  n-k+2\right)  \cdots n} & \cdots & \Xi_{n}^{\left(
n-k+2\right)  \cdots n}%
\end{array}
\right)  \medskip\\
&  =\left(  \Xi_{i}^{I}\right)  _{i\in\left\{  1,\cdots,n\right\}  }%
^{I\in\mathcal{T}^{k-1}}=\left(
\begin{array}
[c]{c}%
\Xi^{1\cdots\left(  k-1\right)  }\\
\vdots\\
\Xi^{\left(  n-k+2\right)  \cdots n}%
\end{array}
\right)  =\left(  \Xi_{1},\cdots,\Xi_{n}\right).
\end{align*}
We define a linear map $\pi:\mathbb{R}^{\tbinom{n}{k-1}\times n}\rightarrow
\Lambda^{k}\left(  \mathbb{R}^{n}\right)  $ in the following way%
\begin{align*}
 \pi\left(  \Xi\right)  &=\sum_{i=1}^{n}\Xi_{i}\wedge
e^{i} ,
\end{align*}
where%
\[
\Xi_{i}=\sum_{1\leq i_{1}<\cdots<i_{k-1}\leq n}\Xi_{i}^{i_{1}\cdots i_{k-1}%
}\,e^{i_{1}}\wedge\cdots\wedge e^{i_{k-1}}=\sum_{I\in\mathcal{T}^{k-1}}%
\Xi_{i}^{I}\,e^{I}.
\]

\end{definition}
\begin{remark}
 Observe that this projection map can also be written as,
 \begin{displaymath}
  \pi(\Xi) = \sum_{I \in \mathcal{T}^k} \left( \sum_{j \in I} \operatorname*{sgn}(j, I_j) \Xi_j^{I_j} \right) e^I,
 \end{displaymath}
see 3(vii) in Section \ref{notations} for the notations.
\end{remark}
\begin{remark}
\begin{enumerate}
 \item Note that the map $\pi:\mathbb{R}^{\tbinom{n}{k-1}\times n}\rightarrow
\Lambda^{k}\left(  \mathbb{R}^{n}\right) $ is onto.
\item It is easy to see that $\pi:\mathbb{R}^{n\times n}\rightarrow\Lambda
^{2}\left(  \mathbb{R}^{n}\right)  $ is given by%
\[
\pi\left(  \xi\right)  =\sum_{i=1}^{n}\xi_{i}\wedge e^{i}=\sum_{1\leq i<j\leq
n}\left(  \xi_{j}^{i}-\xi_{i}^{j}\right)  e^{i}\wedge e^{j}, %
\]
so that, with abuse of notation,
\[
\pi\left(  \xi\right) = \xi -\xi^{T} = 2 \left( \frac{\xi -\xi^{T}}{2} \right).
\]
So for $k=2$, $\pi$ is just twice the alternating projection for $2$-tensors (or twice the skew-symmetric projection for square matrices).

\end{enumerate}
\end{remark}

\section{Main theorem and consequences}
\subsection{Main theorem}
The main result of the article is the following:
\begin{theorem}
\label{Prop ext quasi implique quasi}Let $2\leq k\leq n$,  %
 $f:\Lambda^{k} \rightarrow\mathbb{R}$ %
 and $\pi:\mathbb{R}^{\tbinom{n}{k-1}\times n}\rightarrow
\Lambda^{k}$
be the projection map. Then the following equivalences hold%
\begin{align*}
f\text{ ext. one convex }&\Leftrightarrow\text{ }f\circ\pi\text{ rank one
convex. }\\
f\text{ ext. quasiconvex }&\Leftrightarrow\text{ }f\circ\pi\text{ quasiconvex. }\\
f\text{ ext. polyconvex }&\Leftrightarrow\text{ }f\circ\pi\text{ polyconvex. }
\end{align*}
\end{theorem}

\begin{remark}
\textbf{(i)} Note that the theorem does not say that any quasiconvex or rank one convex function
$\phi:\mathbb{R}^{\tbinom{n}{k-1}\times n}\rightarrow\mathbb{R}$ is of the
form $f\circ\pi$ with $f$ ext. quasiconvex or ext. one convex as the following
example shows. We let $n=k=2$  and%
\[
\phi\left(  \Xi\right)  =\det\Xi
\]
which is clearly polyconvex (and thus quasiconvex and rank one convex).
But there is no function
$f:\Lambda^{k}\rightarrow\mathbb{R}$ such that
$\phi=f\circ\pi.$ Indeed if such an $f$ exists, we arrive at a contradiction, since setting

\[
X=\left(
\begin{array}
[c]{cc}%
1 & 0\\
0 & 1
\end{array}
\right)  \quad\text{and}\quad Y=\left(
\begin{array}
[c]{cc}%
0 & 0\\
0 & 0
\end{array}
\right) ,
\]
we have $\pi\left(  X\right)  =\pi\left(  Y\right)  =0$ and thus%
\[
1=\phi\left(  X\right)  =f\left(  \pi\left(  X\right)  \right)  =f\left(
\pi\left(  Y\right)  \right)  =\phi\left(  Y\right)  =0.
\]

\textbf{(ii)} The following equivalence is, of course, trivially true%
\[
f\text{ convex }\Leftrightarrow\text{ }f\circ\pi\text{ convex.}%
\]

\end{remark}
\subsection{Relations between notions of exterior convexity}
We now list a few simple consequences of the main theorem.
\begin{theorem}
\label{Thm general sur poly quasi ...}Let $1\leq k\leq n$ and $f:\Lambda
^{k} \rightarrow\mathbb{R}.$ Then%
\[
f\text{ convex }\Rightarrow\text{ }f\text{ ext. polyconvex }\Rightarrow\text{
}f\text{ ext. quasiconvex }\Rightarrow\text{ }f\text{ ext. one convex.}%
\]
\end{theorem}
\begin{proof}
 The result is immediate from theorem \ref{Prop ext quasi implique quasi} and the classical results  (cf. \cite{DCV2}). Another,
 more direct proof, without using the classical results, can be found in \cite{BDS1}.
\end{proof}
\begin{theorem}
\label{Thm principal quasiaffine}Let $1\leq k\leq n$ and $f:\Lambda^{k}\rightarrow\mathbb{R}.$ The following statements are
then equivalent.\smallskip

\textbf{(i)} $f$ is ext. polyaffine.\smallskip

\textbf{(ii)} $f$ is ext. quasiaffine.\smallskip

\textbf{(iii)} $f$ is ext. one affine.\smallskip

\textbf{(iv)} For every $0\leq s\leq\left[
n/k\right]$, there exist $c_{s}\in\Lambda^{ks}$ such that,
\[
f\left(  \xi\right)  =\sum_{s=0}^{\left[  n/k\right]  }\left\langle c_{s}%
;\xi^{s}\right\rangle,\text{ for every }\xi\in\Lambda^{k}.
\]

\end{theorem}
\begin{proof}
 From the definitions of ext. polyaffine functions, it is clear that
 \[
(i)\;\Leftrightarrow\;(iv).
\]
The statements%
\[
(i)\;\Leftrightarrow\;(ii)\;\Leftrightarrow\;(iii)
\]
follow at once from  classical results (cf.\ Theorem $5.20$ in \cite{DCV2}) by virtue of Theorem \ref{Prop ext quasi implique quasi}.
\end{proof}\smallskip

For a more direct proof of the above, see \cite{BDS1}. See also \cite{BS2} for yet another proof.
\begin{theorem}
Let $1\leq k\leq n,$ $1 < p < \infty,$ $\Omega
\subset\mathbb{R}^{n}$ be a bounded smooth open set and $f:\Lambda^{k}  \rightarrow\mathbb{R}$ be ext. quasiconvex verifying,
for every $\xi\in\Lambda^{k},$%
\[
c_{1}\left(  \left\vert \xi\right\vert ^{p}-1\right)  \leq f\left(
\xi\right)  \leq c_{2}\left(  \left\vert \xi\right\vert ^{p}+1\right)
\]
for some $c_{1}\,,c_{2}>0.$
if%
\[
\alpha_{s}\rightharpoonup\alpha\quad\text{in }W^{1,p}\left(  \Omega
;\Lambda^{k-1}\right)
\]
then%
\[
\underset{s\rightarrow\infty}{\lim\inf}\int_{\Omega}f\left(  d\alpha
_{s}\right)  \geq\int_{\Omega}f\left(  d\alpha\right)  .
\]
\end{theorem}
\begin{proof}
According to Theorem \ref{Prop ext quasi implique quasi}, we have that
$f\circ\pi$ is quasiconvex. Then classical results (see Theorem 8.4 in
\cite{DCV2}) show that%
\[
\underset{s\rightarrow\infty}{\lim\inf}\int_{\Omega}f\left(  d\alpha
_{s}\right)  =\underset{s\rightarrow\infty}{\lim\inf}\int_{\Omega}f\left(
\pi\left(  \nabla\alpha_{s}\right)  \right)  \geq\int_{\Omega}f\left(
\pi\left(  \nabla\alpha\right)  \right)  =\int_{\Omega}f\left(  d\alpha
\right)
\]
as wished.
\end{proof}
\section{Algebraic properties of the projection}
We now start exploring the algebraic structure of the projection map in greater detail.
The following properties are easily obtained. See \cite{sil} for a proof.

\begin{proposition}\label{Prop. proj rang un et quasi}
Let $2\leq k\leq n$ and $\pi:\mathbb{R}^{\tbinom{n}{k-1}\times n}%
\rightarrow\Lambda^{k}\left(  \mathbb{R}^{n}\right)  $ be the projection map.\smallskip

(i) If $\alpha\in\Lambda^{k-1}\left(  \mathbb{R}^{n}\right)  \sim
\mathbb{R}^{\tbinom{n}{k-1}}$ and $\beta\in\Lambda^{1}\left(  \mathbb{R}%
^{n}\right)  \sim\mathbb{R}^{n},$ then,%
\[
\pi\left(  \alpha\otimes\beta\right)  =\alpha\wedge\beta.
\]

(ii) Let $\omega\in C^{1}\left(  \Omega;\Lambda^{k-1}\right)  ,$ then, by
abuse of notations,%
\[
\pi\left(  \nabla\omega\right)  =d\omega.
\]
\end{proposition}

The following result is crucial to establish the main theorem in the case of polyconvexity. See section 5.4 of \cite{DCV2} for the definition of adjugates and
 section \ref{notations} for the notations.

\begin{proposition}[Adjugate formula]\label{formula} If $k$ is even, then for $2 \leq s \leq \left[ n/k \right]$,
\begin{equation*}
\left[ \pi (\Xi)\right]^s = (s!) \sum_{I \in
\mathcal{T}^{sk}} \left( \sum\nolimits^I_s
\operatorname*{sgn}(J;\tilde{I}) (\operatorname*{adj}\nolimits_s \Xi)^{\tilde{I}}_{J} \right) e^I ,
\end{equation*}
 and \begin{equation*} \left[ \pi (\Xi)\right]^s = 0,  \text{ \quad for } \left[ n/k \right] < s \leq \min\left\{  n, \tbinom{n}{k-1} \right\} . \end{equation*}
 If $k$ is odd,
 \begin{equation*} \left[ \pi (\Xi)\right]^s = 0, \text{ \quad for all } s,\  2 \leq s \leq \min\left\{  n, \tbinom{n}{k-1} \right\} . \end{equation*}
 \end{proposition}\smallskip

\begin{proof}
 We prove only the first equality. Everything else follows by properties of the wedge power.
 So we prove the case when $k$ is even and $2 \leq s \leq \left[ n/k \right]$.
 We prove it by induction.\\
 \emph{Step 1:} To start the induction, we first prove the case when $s=2$. \\
 We have,
 \begin{equation*}
  \pi(\Xi) = \sum_{I \in \mathcal{T}^k} \left( \sum_{j \in I} \operatorname*{sgn}(j, I_j) \Xi_j^{I_j} \right) e^I.
 \end{equation*}
So,
\begin{align*}
 &[\pi(\Xi)]^2 = \pi(\Xi)\wedge \pi(\Xi)  \\
 &=\sum_{I \in \mathcal{T}^{2k}}\left(
 \begin{aligned}
 \sum_{\substack{I^1,\  I^2\\
        I^1 \cup I^2 = I \\
        I^1 \cap I^2 = \emptyset}}\operatorname*{sgn}\left(I^1,I^2\right)  \left(\sum_{j_1 \in I^1} \right.
                                        &\left. \operatorname*{sgn}\left(j_1, I^1_{j_1}\right)   \Xi_{j_1}^{I^1_{j_1}}\right)  \\
& \left(\sum_{j_2 \in I^2}\operatorname*{sgn}\left(j_2, I^2_{j_2}\right)\Xi_{j_2}^{I^2_{j_2}}\right)
                                       \end{aligned}  \right) e^I
                                       \end{align*}

Now, since $k$ is even, we have,
 \begin{align*}
 &[\pi(\Xi)]^2  \\
 &\begin{aligned} =  2    \sum_{I \in \mathcal{T}^{2k}} \left( \rule{0in}{20pt} \right.  \sum\nolimits^I_2
                                         \left( \right. & \operatorname*{sgn}([j_1, I^1_{j_1}],[j_2, I^2_{j_2}])
                                         \operatorname*{sgn}(j_1,I^1_{j_1})\operatorname*{sgn}(j_2,I^2_{j_2}) \Xi_{j_1}^{I^1_{j_1}}\Xi_{j_2}^{I^2_{j_2}}
                                                                                 \\   &+ \operatorname*{sgn}([j_1,I^2_{j_2}],[j_2, I^1_{j_1}])
                                         \operatorname*{sgn}(j_1,I^2_{j_2})\operatorname*{sgn}(j_2,I^1_{j_1}) \Xi_{j_2}^{I^1_{j_1}}\Xi_{j_1}^{I^2_{j_2}}
                                          \left. \right) \left. \rule{0in}{20pt} \right) e^I \end{aligned}\\
     &= 2 \sum_{I \in \mathcal{T}^{2k}} \left( \sum\nolimits^I_2  ( \operatorname*{sgn}(j_1,I^1_{j_1},j_2, I^2_{j_2})
     \Xi_{j_1}^{I^1_{j_1}}\Xi_{j_2}^{I^2_{j_2}} + \operatorname*{sgn}(j_1,I^2_{j_2},j_2, I^1_{j_1})\Xi_{j_2}^{I^1_{j_1}}\Xi_{j_1}^{I^2_{j_2}} ) \right)e^I \\
      &=2\sum_{I \in \mathcal{T}^{2k}} \left( \sum\nolimits^I_2 \operatorname*{sgn} (j_1,I^1_{j_1},j_2, I^2_{j_2})(
                                         \Xi_{j_1}^{I^1_{j_1}}\Xi_{j_2}^{I^2_{j_2}} - \Xi_{j_2}^{I^1_{j_1}}\Xi_{j_1}^{I^2_{j_2}})\right) e^I \\
      &= 2\sum_{I \in \mathcal{T}^{2k}} \left( \sum\nolimits^I_2 \operatorname*{sgn} (j_1,I^1_{j_1},j_2, I^2_{j_2})
                                         (\operatorname*{adj}\nolimits_2\Xi)_{j_1j_2}^{I^1_{j_1}I^2_{j_2}}\right) e^I ,
     \end{align*}
which proves the case for $s=2$. \\
\emph{Step 2:}
We assume the result to be true for some $s \geq 2$ and show that it holds for $s+1$, thus completing the induction.
Now we know, by Laplace expansion formula for the determinants,
\begin{align*}
(\operatorname*{adj}\nolimits_{s+1} \Xi)^{I^1\ldots I^{s+1}}_{j_1\ldots j_{s+1}} &= \sum_{m=1}^{s+1} \Xi_{j_l}^{I^m} (-1)^{l+m}
(\operatorname*{adj}\nolimits_{s} \Xi)^{I^1\ldots \widehat{I^m}\ldots I^{s+1}}_{j_1\ldots \widehat{j_l}\ldots j_{s+1}} , \textrm{ for any } 1 \leq l \leq  s+1\\
 &= \frac{1}{s+1}\sum_{l=1}^{s+1}\sum_{m=1}^{s+1} \Xi_{j_l}^{I^m} (-1)^{l+m}
(\operatorname*{adj}\nolimits_{s} \Xi)^{I^1\ldots \widehat{I^m}\ldots I^{s+1}}_{j_1\ldots \widehat{j_l}\ldots j_{s+1}} .
\end{align*}
Note that,  for any  $1 \leq l,m \leq  s+1$,
\begin{equation*}
 \operatorname*{sgn}(j_1,I^1,\ldots, j_{s+1}, I^{s+1}) = (-1)^{\{ (l-1)+(m-1)(k-1) \}}\operatorname*{sgn}(j_l, I^m,\tilde{I}^{l,m}),
\end{equation*}
where $\tilde{I}^{l,m}$ is a shorthand for the permutation $(\tilde{j}_1, \tilde{I}^1,\ldots, \tilde{j}_s, \tilde{I}^s)$
and
\begin{itemize}
 \item  $\tilde{j}_1 <  \ldots < \tilde{j}_s$ and $\{ \tilde{j}_1 , \ldots , \tilde{j}_s\} = \{ j_1, \ldots,
\widehat{j_l}, \ldots, j_{s+1} \}. $
\item $\tilde{I}^1 < \ldots < \tilde{I}^s$ and
$\{ \tilde{I}^1 , \ldots , \tilde{I}^s\} = \{ I^1, \ldots,
\widehat{I^m}, \ldots, I^{s+1} \}$.
\end{itemize}
Note that this means $\tilde{j}_r = j_r$
for $1 \leq r < l$ and $\tilde{j}_r = j_{r+1}$ for $l \leq r
\leq s$. Similarly, $\tilde{I}^r = I^r$ for $1 \leq r < m$ and $\tilde{I}^r = I^{r+1}$ for $m \leq r \leq s$.
Now since $k$ is even, for any  $1 \leq l,m \leq  s+1$,
\begin{align*}
 &\operatorname*{sgn}(j_1, I^1,\ldots, j_{s+1}, I^{s+1}) =  (-1)^{l+m}\operatorname*{sgn}(j_l, I^m,\tilde{I}^{l,m}) \\
 &=(-1)^{l+m}\operatorname*{sgn}(j_l, I^m)\operatorname*{sgn}(\tilde{I}^{lm})\operatorname*{sgn}([j_l, I^m],[\tilde{I}^{l,m}]).
 \end{align*}
Thus,
\begin{align*}
&\operatorname*{sgn}(j_1,  I^1,\ldots, j_{s+1}, I^{s+1})(\operatorname*{adj}\nolimits_{s+1} \Xi)^{I^1\ldots I^{s+1}}_{j_1\ldots j_{s+1}} \\
&=\frac{1}{(s+1)}\sum_{l,m =1}^{s+1}\operatorname*{sgn}([j_l, I^m],[\tilde{I}^{l,m}])\operatorname*{sgn}(j_l, I^m)\Xi_{j_l}^{I^m}
\operatorname*{sgn}(\tilde{I}^{lm})(\operatorname*{adj}\nolimits_{s} \Xi)^{I^1\ldots \widehat{I^m}\ldots I^{s+1}}_{j_1\ldots \widehat{j_l}\ldots j_{s+1}}.
\end{align*}
Hence,
{\allowdisplaybreaks
\begin{align*}
& (s+1)! \sum_{I \in
\mathcal{T}^{(s+1)k}} \left( \sum\nolimits^I_{s+1}
\operatorname*{sgn}(j_1,I^1,\ldots , j_{s+1}, I^{s+1}) (\operatorname*{adj}\nolimits_{s+1} \Xi)^{I^1\ldots I^{s+1}}_{j_1\ldots j_{s+1}} \right) e^I\\
&=\frac{(s+1)!}{(s+1)} \sum_{I \in
\mathcal{T}^{(s+1)k}} \left( \rule{0in}{40pt}
\right. \begin{aligned} \sum\nolimits^I_{s+1}\sum_{l,m =1}^{s+1}\operatorname*{sgn}([j_l, I^m],[\tilde{I}^{l,m}])
\operatorname*{sgn}(j_l, I^m)\Xi_{j_l}^{I^m} \\
\operatorname*{sgn}(\tilde{I}^{lm})(\operatorname*{adj}\nolimits_{s} \Xi)^{I^1\ldots
\widehat{I^m}\ldots I^{s+1}}_{j_1\ldots \widehat{j_l}\ldots j_{s+1}}  \end{aligned}\left.\rule{0in}{40pt}  \right) e^I \\
&=(s!)\sum_{I \in
\mathcal{T}^{(s+1)k}} \left( \rule{0in}{40pt} \right. \begin{aligned} \sum\nolimits^I_{s+1}\sum_{l,m =1}^{s+1}\operatorname*{sgn}([j_l, I^m],[\tilde{I}^{l,m}])
\operatorname*{sgn}(j_l, I^m)\Xi_{j_l}^{I^m} \\
\operatorname*{sgn}(\tilde{I}^{lm})(\operatorname*{adj}\nolimits_{s} \Xi)^{I^1\ldots
\widehat{I^m}\ldots I^{s+1}}_{j_1\ldots \widehat{j_l}\ldots j_{s+1}}  \end{aligned}\left. \rule{0in}{40pt} \right) e^I \\
&=\sum_{I \in
\mathcal{T}^{(s+1)k}} \left( \rule{0in}{40pt} \right.  \begin{aligned} \sum_{\substack{I' \subset I\\ I' \in \mathcal{T}^k}}& \left( \rule{0in}{20pt} \right. \operatorname*{sgn}  (I',[I \backslash I'])
(\sum_{j \in I'} \operatorname*{sgn}(j, I'_j)\Xi_{j}^{I'_j}) \\  & \left( s! \left( \sum\nolimits^{[I\backslash I']}_s
\operatorname*{sgn}(\tilde{j}_1,\tilde{I}^1,\ldots , \tilde{j}_s, \tilde{I}^s)
(\operatorname*{adj}\nolimits_s \Xi)^{\tilde{I}^1\ldots \tilde{I}^s}_{\tilde{j}_1\ldots \tilde{j}_s} )\right) \right) \end{aligned}
\left.\rule{0in}{40pt} \right) e^I ,
\end{align*}
}

where the last line is just a rewriting of the penultimate one. Indeed on expanding the sums the map, sending $j_{l}$ to $j$;
$I^{m}$ to $I'_{j}$; $I^{1}, \ldots,\widehat{I^{m}}, \ldots , I^{s+1}$ to $\tilde{I}^{1}, \ldots ,\tilde{I}^{s}$ respectively and
$j_1, \ldots ,\widehat{j_l} ,\ldots ,j_{s+1}$ to $\tilde{j}_{1}, \ldots , \tilde{j}_{s}$ respectively is a bijection
between the terms on the two sides of
 the last equality.\smallskip

So, we have, by induction hypothesis,
{\allowdisplaybreaks
\begin{align*}
& (s+1)! \sum_{I \in
\mathcal{T}^{(s+1)k}} \left( \sum\nolimits^I_{s+1}
\operatorname*{sgn}(j_1,I^1,\ldots , j_{s+1}, I^{s+1}) (\operatorname*{adj}\nolimits_{s+1} \Xi)^{I^1\ldots I^{s+1}}_{j_1\ldots j_{s+1}} \right) e^I \nonumber \\
&=\sum_{I \in
\mathcal{T}^{(s+1)k}} \left( \rule{0in}{40pt} \right.
\begin{aligned} \sum_{\substack{I' \subset I\\ I' \in \mathcal{T}^k}} &
\operatorname*{sgn}  (I',[I \backslash I'])
  \left( \rule{0in}{20pt}   \text{ coefficient of } e^{I'} \text{ in }
 \pi (\Xi) \right) \\
  & \times \left( \rule{0in}{20pt} \right. \text{ coefficient of } e^{[I \backslash I']} \text{ in }
  \left[ \pi (\Xi)\right]^s  \left.\rule{0in}{20pt} \right)
\end{aligned}
 \left. \rule{0in}{40pt} \right) e^I \\
 &= \sum_{I \in
\mathcal{T}^{(s+1)k}} \left( \text{ coefficient of } e^I \text{ in } \left[ \pi (\Xi)\right]^{s+1} \right) e^I =  \left[ \pi (\Xi)\right]^{s+1} .
\end{align*}
}
This completes the induction proving the desired result.
\end{proof}
\par Since we have seen that $\left[ \pi (\Xi)\right]^s$ depends only on
$\operatorname*{adj}\nolimits_s \Xi$, we are now in a position to define a
linear projection for every value of $s$. These maps will be useful later.
\begin{definition}
 For every $2 \leq s \leq  \min\left\{  n, \tbinom{n}{k-1} \right\}$, we define the linear projection
 maps
 $\pi_s:\mathbb{R}^{\tbinom{\tbinom{n}{k-1}}{s}\times \tbinom{n}{s}}\rightarrow  \Lambda^{ks}(\mathbb{R}^n)$ by the condition,
 \begin{equation*}
  \pi_s(\operatorname*{adj}\nolimits_s(\Xi)) = \left[ \pi (\Xi)\right]^s \text{ for all } \Xi \in \mathbb{R}^{\tbinom{n}{k-1}\times n}.
 \end{equation*}
\end{definition}
\begin{remark}
 It is clear that this condition uniquely defines the projection maps. For the sake of consistency, we define,
$ \pi_1 = \pi $ and $\pi_0 $ is defined to be the identity map from $\mathbb{R}$ to $\mathbb{R}$.
\end{remark}

\section{An important lemma}

 \begin{lemma}\label{polyconvexitylemma}
  Let $2 \leq k \leq n$ and $N = \tbinom{n}{k-1} $. Consider the function
  \begin{equation*}
   g(X,d)=f(\pi(X)) - \sum_{s=1}^{\min\left\{  N, n  \right\}} \left\langle d_s, \operatorname*{adj}\nolimits_s X \right\rangle
  \end{equation*}
  where $d = (  d_1,\ldots, d_{\min\left\{  N, n \right\}} )$, $d_s \in \mathbb{R}^{\tbinom{N}{s} \times \tbinom{n}{s}}$ for
  all $1 \leq s \leq \min\left\{  N, n  \right\}$ and $X \in \mathbb{R}^{N \times n}$.
  If for a given vector $d$, the function $X \mapsto g(X,d)$ achieves a minimum over $\mathbb{R}^{N \times n}$, then for all
  $ 1 \leq s   \leq \min\left\{  N, n  \right\}$, there exists $\mathcal{D}_{s} \in \Lambda^{ks}$ such that,
  \begin{equation*}
   \langle d_s , \operatorname*{adj}\nolimits_{s} Y \rangle = \langle \mathcal{D}_s , \pi_s(\operatorname*{adj}\nolimits_{s} Y) \rangle \text{ \quad for all } Y \in \mathbb{R}^{N \times n}.
  \end{equation*}
 \end{lemma}\smallskip

 The lemma is technical and quite heavy in terms of notations.  So before proceeding to prove the lemma as stated,
 it might be helpful to spell out the idea of the proof. The plan is always the same. In short, if the conclusion of the lemma does not hold, we can always choose a matrix $X$ such
 that $g(X,d)$ can be made to be smaller than any given real number, contradicting the hypothesis that the map $X \mapsto g(X,d)$ assumes a minimum. \smallskip

\begin{proof}
Let us fix a vector $d$ and assume that for this $d$, the function $X \mapsto g(X,d)$ achieves a minimum over $\mathbb{R}^{N \times n}$.

We will first show that all adjugates with a common index between subscripts and superscripts must have zero coefficients. More precisely, we claim that,
\begin{claim}\label{common index have zero coefficients k even}
 For any $2 \leq k \leq n$ and for every $1 \leq s \leq \min\left\{  N, n  \right\}$, for every
 $J\in \mathcal{T}^{s}, I = \left\lbrace I^{1}\ldots I^{s} \right\rbrace$ where $ I^{1},\ldots,I^{s} \in \mathcal{T}^{k-1}$,we have,
 \begin{equation*}
   \left( d_s \right)^{I}_{J} = 0 \textrm{ whenever  }
  I\cap J \neq \emptyset.
  \end{equation*}
\end{claim}

We prove claim \ref{common index have zero coefficients k even},  using induction over $s$.
 To start the induction, we first show the case $s=1$. Let $j \in I$, where $I \in \mathcal{T}^{k-1}$. We choose $X = \lambda e^{j}\otimes e^{I} $,
 then clearly $\pi(X) = 0.$
Also, $g(X,d) = f(0) - \lambda \left(d_1 \right)_{j}^{I}$. By letting $\lambda$ to $+\infty$ and $-\infty$ respectively, we deduce that
$\left(d_1 \right)_{j}^{I}= 0$, since otherwise we obtain a contradiction to the fact that $g$ achieves a finite minima.\smallskip

Now we assume that claim \ref{common index have zero coefficients k even} holds for all $ 1 \leq s \leq p $ and prove the result for $s= p+1.$
We consider $ \left( d_{p+1} \right)^{ I^{1} \ldots  I^{p+1} }_{  j_{1} \ldots  j_{p+1}  }$ with
 $   j_l \in I^m \textrm{ for some } 1 \leq l,m \leq p+1.$

  Now we first order the rest of the indices (other than the common index) in subscripts and the rest of the multiindices (other
  than the one with the common index) in superscripts.
  Let $\tilde{I}^{1} <  \ldots < \tilde{I}^{p}$ and $ \tilde{j}_{1}<  \ldots < \tilde{j}_{p}$
   represent the multiindices and indices in the sets $\left\lbrace  I^{1}, \ldots , I^{p+1} \right\rbrace\setminus \lbrace I^{m}\rbrace $
   and $\left\lbrace  j_{1}, \ldots , j_{p+1}\right\rbrace\setminus \lbrace j_{l}\rbrace  $ respectively.

   Now we choose, \begin{equation*}
                   X = \lambda e^{j_{l}}\otimes e^{I^{m}} + \sum_{r=1}^{p} e^{\tilde{j}_r}\otimes e^{\tilde{I}^r}.
                  \end{equation*}
  Since $j_{l} \in I^{m}$,    we get $\pi(X)$ is independent of $\lambda$. Also, all lower order non-constant adjugates of $X$ must contain the
   index  $ j_{l}$ both in subscript and in superscript and hence their coefficients are $0$ by the induction hypothesis. Hence, the only
   non-constant adjugate of $X$ appearing in the expression for $g(X,d)$ is,
   $$\left( \operatorname*{adj}\nolimits_{p+1} X \right)^{I^{1}  \ldots  I^{p+1} }_{j_{1}  \ldots  j_{p+1} } = \left( -1 \right)^{\alpha}\lambda,$$ where
   $\alpha$ is a fixed integer. Now,
   $$g(X,d) = \left( -1 \right)^{\alpha+1}\lambda \left( d_{p+1} \right)^{ I^{1}  \ldots  I^{p+1} }_{  j_{1}  \ldots  j_{p+1}  } + \textrm{ constants }.$$
   Again as before, we let $\lambda$ to $+\infty$ and $-\infty$ and  we deduce, by the same argument,
   $\left( d_{p+1} \right)^{ I^{1} \ldots  I^{p+1} }_{  j_{1} \ldots  j_{p+1}  } =0.$ This completes the induction and proves the claim.\smallskip

 At this point we split the proof in two cases, the case when $k$ is an even integer and the case when $k$ is an odd integer.\bigskip

 \textbf{Case 1: $\mathbf{ k}$ is even }\smallskip

Note that, unless $k=2$, it does not follow from above that $d_s = 0$ for all $s \geq [\frac{n}{k}]$. The possibility that
two different blocks of multiindices in the superscript have some index in common has not been ruled out. Now we will show that the
coefficients of two different adjugates having the same set of indices are related in the following way:

   \begin{claim}\label{interchange of indices k even}
    For every $ s  \geq 1$,
     \begin{equation*}
     \operatorname*{sgn}(J;I)\left( d_s \right)^{I}_{J}
    = \operatorname*{sgn}(\tilde{J}; \tilde{I})\left( d_s \right)^{ \tilde{I}}_{ \tilde{J}},
    \end{equation*}
     whenever $  J \cup I = \tilde{J} \cup \tilde{I} ,$ with $J, \tilde{J} \in \mathcal{T}^{s}$ ,
     $I = \left\lbrace I^{1}\ldots I^{s}\right\rbrace = \left[ I^{1},  \ldots ,I^{s}\right]$,
     $\tilde{I} = \left\lbrace \tilde{I}^{1}\ldots \tilde{I}^{s}\right\rbrace= [ \tilde{I}^{1}, \ldots , \tilde{I}^{s}]$,
      $I^{1},\ldots ,I^{s},\tilde{I}^{1},\ldots ,\tilde{I}^{s} \in \mathcal{T}^{k-1}$  and $J\cap I = \emptyset.$ In particular, given any
      $U \in \mathcal{T}^{ks}$, there exists a constant
      $\mathcal{D}_{U} \in \mathbb{R}$ such that,
      \begin{equation}\label{DI}
       \operatorname*{sgn}(J;I)\left( d_s \right)^{I}_{J}= \mathcal{D}_{U},
      \end{equation}
      for all $J \cup I =  U$ with $J \in \mathcal{T}^{s}$ ,
     $I = \left\lbrace I^{1}\ldots I^{s}\right\rbrace = \left[ I^{1},  \ldots ,I^{s}\right]$,
    $I^{1},\ldots ,I^{s} \in \mathcal{T}^{k-1}$.
   \end{claim}
   We will prove the claim again by induction over $s$. We first prove it for the case $s=1.$\smallskip

   For the case $s=1$, we just need to prove, for any index $j$, any multindex $I \in \mathcal{T}^{k-1}$ such that $j \cap I = \emptyset$, we have
   \begin{equation}\label{s equal to one k even}
    \operatorname*{sgn}(j,I)\left( d_1 \right)^{ I}_{  j }
    = \operatorname*{sgn}(\tilde{j},\tilde{I})\left( d_1 \right)^{ \tilde{I}}_{\tilde{j}} ,
   \end{equation}
 where $[j,I]=[\tilde{j},\tilde{I}].$
    We choose $X = \lambda  \operatorname*{sgn}(j,I) e^{j}\otimes e^{I} - \lambda \operatorname*{sgn}(\tilde{j},\tilde{I}) e^{\tilde{j}}\otimes e^{\tilde{I}}.$ Clearly, $\pi(X)= 0$ and this gives,
    \begin{equation*}
     g(X,d) = f(0) + \lambda \left(     \operatorname*{sgn}(j,I) \left(  d_1 \right)^{ I}_{  j } -  \operatorname*{sgn}(\tilde{j},\tilde{I}) \left( d_1 \right)^{\tilde{I}}_{\tilde{j}} \right),
    \end{equation*}
    where we have used claim \ref{common index have zero coefficients k even} to deduce that $\left(d_2\right)^{[I\tilde{I}]}_{[j\tilde{j}]}=0$.
Letting $\lambda$ to $+\infty$ and $-\infty$, we get (\ref{s equal to one k even}).\smallskip

   Now we assume the result for all $1 \leq s \leq s_{0}$ and show it for $s= s_{0}+1.$
   Let $ [ I^{1} \ldots  I^{s_{0}+1}  j_{1}  \ldots  j_{s_{0}+1} ] =
   [ \tilde{I}^{1}  \ldots  \tilde{I}^{s_{0}+1} \tilde{j}_{1}  \ldots  \tilde{j}_{s_{0}+1}] $. Note that the sets
  $\left\lbrace I^{1}  \ldots  I^{s_{0}+1}  j_{1}  \ldots  j_{s_{0}+1}\right\rbrace$ and
  $\left\lbrace\tilde{I}^{1} \ldots  \tilde{I}^{s_{0}+1} \tilde{j}_{1}  \ldots  \tilde{j}_{s_{0}+1}\right\rbrace$
  are permutations of each other, preserving an order relation given by $ j_{1}< \ldots  < j_{s_{0}+1}$,
   $ \tilde{j}_{1}< \ldots  < \tilde{j}_{s_{0}+1}$, $I^{1}<  \ldots  < I^{s_{0}+1}$ and
   $\tilde{I}^{1} < \ldots  < \tilde{I}^{s_{0}+1}$. Thus the aforementioned sets can be related by any permutation (of $k(s_{0}+1)$ indices) that respects this order. Since any such
  permutation is a product of $k$-flips, it is enough to prove the claim in case of
  $k$-flips, cf.\ definition \ref{kflips}.\smallskip

  We now assume $(J ,  I)$ and $(\tilde{J} , \tilde{I})$ are related by a $k$-flip interchanging the subscript $j_l$ with one index in the superscript block $I^m$ and keep all the other indices unchanged. Also,
    we assume that after the interchange, the
    position of the multiindex containing $j_l$ in the superscript is $p$ and  the
    new position of the index from the multiindex $I^m$ in the subscript is $q$, i.e,
    $j_l \in \tilde{I}^p$ and  $\tilde{j}_q \in I^m.$
    We also order the remaining indices and assume ,
    $$\breve{I} = [ \breve{I}^1  , \ldots , \breve{I}^{s_{0}}] = \lbrace \breve{I}^1    \ldots  \breve{I}^{s_{0}}\rbrace
    = \left\lbrace I^{1}  \ldots \widehat{I^{m}} \ldots I^{s_{0}+1}  \right\rbrace , $$
    and
    $$\breve{J}= [\breve{j}_1 \ldots  \breve{j}_{s_{0}}]=\lbrace\breve{j}_1   \ldots \breve{j}_{s_{0}}\rbrace
    =\left\lbrace  j_{1}  \ldots \widehat{j_{l}} \ldots j_{s_{0}+1}\right\rbrace $$ respectively.
    Now we choose,
    \begin{equation*}
     X = \lambda  \operatorname*{sgn}(j_{l} , I^{m}) e^{j_l}\otimes e^{I^m} - \lambda \operatorname*{sgn}(\tilde{j}_{q} ,
     \tilde{I}^{p}) e^{\tilde{j}_{q}}\otimes e^{\tilde{I}^{p}} + \sum_{1 \leq r \leq s_{0}} e^{\breve{j}_r}\otimes e^{\breve{I}_r}.
    \end{equation*}

    Note that $\pi(X)$ is independent of $\lambda$. Also, all non-constant adjugates of $X$ appearing
    with possibly non-zero coefficients in the expression for $g(X,d)$  have, either $j_{l}$ in subscript and $I^{m} $ in superscript or
    has $\tilde{j}_{q} $ as a subscript and  $\tilde{I}^{p} $ as a superscript,
     but never both as then they have zero coefficients by claim \ref{common index have zero coefficients k even}. Also, these adjugates occur in pairs.
     More precisely, for every
    non-constant adjugate of $X$ appearing
    with possibly non-zero coefficients in the expression for $g(X,d)$  having $j_{l}$ in subscript and $I^{m} $ in superscript, there is one
    having $\tilde{j}_{q} $ in subscript and  $\tilde{I}^{p} $ in superscript.\smallskip

    Let us show that, for any $1 \leq s \leq s_{0}+1$, any subset
    $ \bar{J}_{s-1} = \left\lbrace \bar{j}_{1},  \ldots , \bar{j}_{s-1} \right\rbrace
   \subset \breve{J} $ of $s$ indices and
any choice of  of $s-1$ multiindices $ \bar{I}^{1},  \ldots , \bar{I}^{s-1} $ out of $s_{0}$ multiindices
$\breve{I}^1  , \ldots , \breve{I}^{s_{0}}$, we have,

\begin{equation}\label{formula sign adjugate k even}
\frac{\left(\operatorname*{adj}\nolimits_{s} X \right)^{ [ I^{m}, \bar{I}^{1}, \ldots  , \bar{I}^{s-1} ] }_{ [ j_{l}\bar{J}_{s-1} ]  }}
 {\operatorname*{sgn}([ j_{l}\bar{J}_{s-1} ] ;[ I^{m}, \bar{I}^{1},  \ldots  , \bar{I}^{s-1} ] )               }
  = -\frac{\left(\operatorname*{adj}\nolimits_{s} X \right)^{  [ \tilde{I}^{p}, \bar{I}^{1},  \ldots  , \bar{I}^{s-1} ] }_{
  [ \tilde{j}_{q} \bar{J}_{s-1} ]  }}
 {\operatorname*{sgn}( [ \tilde{j}_{q} \bar{J}_{s-1} ] ; [ \tilde{I}^{p}, \bar{I}^{1}, \ldots  , \bar{I}^{s-1} ])} .
\end{equation}

Let $a_1$ be the position of $j_l$ in $\left[ j_{l}\bar{J}_{s-1}\right]$ ,
$a_2$ be the position of $ \tilde{j}_{q}$ in $\left[ \tilde{j}_{q} \bar{J}_{s-1} \right]$, $b_1$ be the position of $ I^{m}$ in
$[ I^{m}, \bar{I}^{1},  \ldots  , \bar{I}^{s-1} ]$ and
$b_2 $ be the position of $\tilde{I}^{p}$ in $\left[  \tilde{I}^{p}, \bar{I}^{1},  \ldots  , \bar{I}^{s-1} \right]$.\smallskip

Since $k$ is even,
\begin{align*}
 \operatorname*{sgn} &([ j_{l}\bar{J}_{s-1} ] ;[ I^{m}, \bar{I}^{1},  \ldots  , \bar{I}^{s-1} ] )  \notag \\
 &=  (-1)^{\lbrace ( a_{1}-1) +(b_{1}-1) \rbrace}
 \operatorname*{sgn}(j_{l},I^{m})\operatorname*{sgn}  ( \bar{J}_{s-1};  \lbrace \bar{I}^{1}  \ldots   \bar{I}^{s-1} \rbrace ) \notag\\
  &\qquad \qquad \qquad  \qquad \operatorname*{sgn}([j_{l}, I^{m}], [( \bar{J}_{s-1}; \lbrace \bar{I}^{1} \ldots  \bar{I}^{s-1} \rbrace ) ]),
\end{align*}
and
\begin{align*}
 \operatorname*{sgn} &([ \tilde{j}_{q}\bar{J}_{s-1} ] ;[ \tilde{I}^{p}, \bar{I}^{1},  \ldots  , \bar{I}^{s-1} ] )  \notag \\
 &=  (-1)^{\lbrace ( a_{2}-1) +(b_{2}-1) \rbrace}
 \operatorname*{sgn}(\tilde{j}_{q},\tilde{I}^{p})\operatorname*{sgn}  ( \bar{J}_{s-1};  \lbrace \bar{I}^{1}  \ldots   \bar{I}^{s-1} \rbrace ) \notag\\
  &\qquad \qquad \qquad  \qquad \operatorname*{sgn}([\tilde{j}_{q}, \tilde{I}^{p}], [( \bar{J}_{s-1};  \lbrace \bar{I}^{1}  \ldots   \bar{I}^{s-1} \rbrace ) ]).
\end{align*}
We also have,
\begin{align*}
 \left(\operatorname*{adj}\nolimits_{s} X \right)^{ [ I^{m}, \bar{I}^{1},  \ldots  , \bar{I}^{s-1} ] }_{ [ j_{l}\bar{J}_{s-1} ]  }
 = (-1)^{a_{1}+b_{1}}\operatorname*{sgn}(j_{l},I^{m}) \lambda
 \left( \operatorname*{adj}\nolimits_{s-1} X \right)^{[ \bar{I}^{1},  \ldots , \bar{I}_{s-1}]}_{[\bar{J}_{s-1}]},
\end{align*}
and
\begin{align*}
 \left(\operatorname*{adj}\nolimits_{s} X \right)^{  [ \tilde{I}^{p}, \bar{I}^{1},  \ldots  , \bar{I}^{s-1} ] }_{
  [ \tilde{j}_{q} \bar{J}_{s-1} ]  }
    =  - (-1)^{a_{2}+b_{2}}\operatorname*{sgn}(\tilde{j}_{q},\tilde{I}^{p}) \lambda
    \left( \operatorname*{adj}\nolimits_{s-1} X \right)^{[ \bar{I}^{1},  \ldots , \bar{I}_{s-1}]}_{[\bar{J}_{s-1}]}.
\end{align*}
Combining the four equations above, the result follows. \smallskip

We now finish the proof of claim \ref{interchange of indices k even}. Using \eqref{formula sign adjugate k even}, we have,
    \begin{align*}
    &g(X,d) = \lambda\left\lbrace \vphantom{\begin{aligned}
            \operatorname*{sgn} &(\left[ j_{l}\bar{J}_{s-1}\right];\left[ i_{m}\bar{I}_{s-1} \right])
    \left( d_{s}\right)^{\left[ i_{m}\bar{I}_{s-1} \right]}_{\left[ j_{l}\bar{J}_{s-1}\right]}  \\
    &-\operatorname*{sgn}(\left[ \tilde{j}_{q} \bar{J}_{s-1} \right];\left[\tilde{i}_{p} \bar{I}_{s-1} \right])
    \left( d_{s}\right)^{\left[\tilde{i}_{p} \bar{I}_{s-1} \right]}_{\left[ \tilde{j}_{q} \bar{J}_{s-1} \right]}
           \end{aligned}} \right.
    (-1)^{\alpha} \left( \operatorname*{sgn}(J;I)\left( d_{s_{0}+1} \right)^{I}_{J}
    - \operatorname*{sgn}(\tilde{J}; \tilde{I})\left( d_{s_{0}+1} \right)^{ \tilde{I}}_{ \tilde{J}}\right) \\
    & + \sum_{s=1}^{s_{0}} \sum\nolimits^{s} 
    k_{s,\gamma}
    \left( \begin{aligned}
            & \operatorname*{sgn} ([ j_{l}\bar{J}_{s-1} ] ;[ I^{m}, \bar{I}^{1},  \ldots  , \bar{I}^{s-1} ] )
    \left( d_{s}\right)^{ [ I^{m}, \bar{I}^{1},  \ldots  , \bar{I}^{s-1} ]}_{\left[ j_{l}\bar{J}_{s-1}\right]}  \\
    &-\operatorname*{sgn} ([ \tilde{j}_{q}\bar{J}_{s-1} ] ;[ \tilde{I}^{p}, \bar{I}^{1}, \ldots  , \bar{I}^{s-1} ] )
    \left( d_{s}\right)^{ [ \tilde{I}^{p}, \bar{I}^{1},  \ldots  , \bar{I}^{s-1} ]}_{\left[ \tilde{j}_{q} \bar{J}_{s-1} \right]}
           \end{aligned} \right)
     \left. \vphantom{\begin{aligned}
            \operatorname*{sgn} &(\left[ j_{l}\bar{J}_{s-1}\right];\left[ i_{m}\bar{I}_{s-1} \right])
    \left( d_{s}\right)^{\left[ i_{m}\bar{I}_{s-1} \right]}_{\left[ j_{l}\bar{J}_{s-1}\right]}  \\
    &-\operatorname*{sgn}(\left[ \tilde{j}_{q} \bar{J}_{s-1} \right];\left[\tilde{i}_{p} \bar{I}_{s-1} \right])
    \left( d_{s}\right)^{\left[\tilde{i}_{p} \bar{I}_{s-1} \right]}_{\left[ \tilde{j}_{q} \bar{J}_{s-1} \right]}
           \end{aligned}}\right\rbrace \\
    &\qquad \qquad \qquad \qquad \qquad \qquad \qquad \qquad \qquad \qquad \qquad \qquad + \textrm{ constants},
 \end{align*}

   where $\sum\nolimits^{s} $ is a shorthand, for every $1 \leq s \leq s_{0}$,  for the sum over all possible such choices
   of $\bar{J}_{s-1}, \bar{I}^{1}, \bar{I}^{2}, \ldots  , \bar{I}^{s-1} $ and $k_{s,\gamma}$ is a generic placeholder for the constants
   appearing before each term of the sum and $\alpha$ is an integer.\smallskip

   By the induction hypothesis, the sum on the right hand side of the above expression is $0$. Hence, we obtain,
   \begin{equation*}
    g(X,d) = (-1)^{\alpha}  \lambda \left( \operatorname*{sgn}(J;I)\left( d_{s_{0}+1} \right)^{I}_{J}
    - \operatorname*{sgn}(\tilde{J}; \tilde{I})\left( d_{s_{0}+1} \right)^{ \tilde{I}}_{ \tilde{J}}\right) + \textrm{ constants }.
    \end{equation*}
Letting $\lambda$ to $+\infty$ and $-\infty$, the claim is proved by induction.\smallskip

   Note that by virtue of claim \ref{interchange of indices k even}, claim \ref{common index have zero coefficients k even} now implies,
that for every $1 \leq s \leq \min\left\{  N, n  \right\}$ , for every
 $J\in \mathcal{T}^{s}, I = \left\lbrace I^{1}\ldots I^{s} \right\rbrace$ where $ I^{1},\ldots,I^{s} \in \mathcal{T}^{k-1}$, we have,
 \begin{equation}\label{distinct indices}
   \left( d_s \right)^{I}_{J} = 0 \textrm{ whenever  either }
  I\cap J \neq \emptyset \textrm{ or } I^{l}\cap I^{m} \neq \emptyset \textrm{ for some } 1 \leq l < m \leq s.
  \end{equation}
 Indeed, if   $I \cap J \neq \emptyset$, we are done, using claim \ref{common index have zero coefficients k even}. So let us assume $I\cap J = \emptyset$ but $ I^{l}\cap I^{m} \neq \emptyset$  for some
 $1 \leq l < m \leq s$. Then there exists an index $i$ such that $i \in I^{l}$ and $i \in I^{m}$, we consider the $k$-flip interchanging
 some index $j$ from subscript with the index $i$ in $I^{l}$. More precisely, let $\tilde{J} \in \mathcal{T}^{s}$
 and $\tilde{I}^{l} \in \mathcal{T}^{k-1}$ be such that $i \in \tilde{J}$,
 $\tilde{J}\setminus \lbrace i \rbrace \subset J$, $I^{l}\setminus \lbrace i \rbrace\subset \tilde{I}^{l}$ and $J\cup I^{l} = \tilde{J}\cup \tilde{I}^{l}$,
 then by claim \ref{interchange of indices k even} we have,
 \begin{equation*}
   \operatorname*{sgn}(J;I)\left( d_s \right)^{I}_{J}
    = \operatorname*{sgn}\left(\tilde{J}; \left[ \tilde{I}^{l},
    I^{1},\ldots,\widehat{I^{l}},\ldots ,I^{s}\right]\right)\left( d_s \right)^{\left[ \tilde{I}^{l}, I^{1},\ldots,\widehat{I^{l}},\ldots , I^{s}\right] }_{ \tilde{J}}.
 \end{equation*}
Since, $i \in \tilde{J}$ and $i \in I^{m}$, $\tilde{J}\cap \left[ \tilde{I}^{l},I^{1},\ldots,\widehat{I^{l}},\ldots , I^{s}\right] \neq \emptyset$, the right hand side
 of above equation is $0$ and so $\left( d_s \right)^{I}_{J} = 0$, which proves \eqref{distinct indices}. So this now implies, $d_s = 0$ for all $s \geq [\frac{n}{k}]$.
 Hence we have, using \eqref{DI}, \eqref{distinct indices} and proposition \ref{formula},
{\allowdisplaybreaks
\begin{align*}
  \langle d_s , \operatorname*{adj}\nolimits_{s} Y \rangle &=   \sum_{I \in
\mathcal{T}^{sk}}  \sum\nolimits^I_s (d_s)^{\tilde{I}}_{J}  ( \operatorname*{adj}\nolimits_{s} Y )^{\tilde{I}}_{J} \\
 &= \sum_{I \in
\mathcal{T}^{sk}}  \sum\nolimits^I_s \operatorname*{sgn}(J;\tilde{I}) (d_s)^{\tilde{I}}_{J}  \operatorname*{sgn}(J;\tilde{I})
( \operatorname*{adj}\nolimits_{s} Y )^{\tilde{I}}_{J} \\
 &= \sum_{I \in
\mathcal{T}^{sk}}  \frac{1}{s!} \mathcal{D}_{I} \sum\nolimits^I_s   (s!) \operatorname*{sgn}(J;\tilde{I})
( \operatorname*{adj}\nolimits_{s} Y )^{\tilde{I}}_{J} \\
 &= \langle \mathcal{D}_{s}, \pi_{s} (\operatorname*{adj}\nolimits_{s} Y ) \rangle ,
\end{align*}
}
where $ \displaystyle \mathcal{D}_{s} = \frac{1}{s!} \sum_{I \in \mathcal{T}^{sk}}  \mathcal{D}_{I} e^{I}$, which finishes the proof when $k$ is even. \bigskip

 \textbf{Case 3: $\mathbf{ k}$ is odd }\smallskip

 In this case, by proposition \ref{formula}, it is enough to show that all coefficients of all terms, except the linear ones must be zero.
 As in the case above, the plan is to establish
  a relation between the coefficients of two different adjugates
  having the same set of indices. But when $k$ is odd, the relationship is not as nice  as in the even case  and as such there is no general formula.
  However, we still have a weaker analogue of claim
  \ref{interchange of indices k even} for the case of $k$-flips.
\begin{claim}\label{interchange of indices k odd}
  For $ s \geq 1 $, if $J, \tilde{J} \in \mathcal{T}^{s}$,  and $I^{1}\ldots ,I^{s},
  \tilde{I}^{1},\ldots ,\tilde{I}^{s} \in \mathcal{T}^{k-1}$, where  $J =  \lbrace j_{1}  \ldots  j_{s}\rbrace$,
  $ \tilde{J} = \lbrace\tilde{j}_{1}  \ldots  \tilde{j}_{s}\rbrace $,
     $I = \left\lbrace I^{1}\ldots I^{s}\right\rbrace = \left[ I^{1}, \ldots ,I^{s}\right]$ and
     $\tilde{I} = \left\lbrace \tilde{I}^{1}\ldots \tilde{I}^{s}\right\rbrace= [ \tilde{I}^{1}, \ldots , \tilde{I}^{s}]$
     be such that $J\cap I = \emptyset$ and $(J , I)$ and $ (\tilde{J} , \tilde{I})$ are related by a $k$-flip
   interchanging an index $j_l$ in the subscript  with one from the multiindex $I^m$ in the superscript. Also,
    we assume that after the interchange, the position of the multiindex containing $j_l$ in the superscript is $p$ and  the
    new position of the index from the multiindex $I^m$ in the subscript is $q$ , i.e ,
    $     j_l \in \tilde{I}^p $ and $  \tilde{j}_q \in I^m.$

    Then we have,

    \begin{equation*}
     \operatorname*{sgn}(J;I)\left( d_s \right)^{I}_{J}
    =  (-1)^{(m-p)}\operatorname*{sgn}(\tilde{J}; \tilde{I})\left( d_s \right)^{ \tilde{I}}_{ \tilde{J}}.
    \end{equation*}
\end{claim}

Since the proof of claim \ref{interchange of indices k odd} is very similar to that of claim \ref{interchange of indices k even},
we shall indicate only a brief sketch of the proof. Since $k$ is odd, we deduce,
\begin{multline*}
 \operatorname*{sgn} ([ j_{l}\bar{J}_{s-1} ] ;[ I^{m}, \bar{I}^{1}, \ldots  , \bar{I}^{s-1} ] )  \\
 =  (-1)^{\lbrace ( a_{1}-1) \rbrace}
 \operatorname*{sgn}(j_{l},I^{m})\operatorname*{sgn}  ( \bar{J}_{s-1};  \lbrace \bar{I}^{1}  \ldots   \bar{I}^{s-1} \rbrace ) \\
  \operatorname*{sgn}([j_{l}, I^{m}], [( \bar{J}_{s-1};  \lbrace \bar{I}^{1}  \ldots   \bar{I}^{s-1} \rbrace ) ]) ,
\end{multline*}
\begin{multline*}
 \operatorname*{sgn} ([ \tilde{j}_{q}\bar{J}_{s-1} ] ;[ \tilde{I}^{p}, \bar{I}^{1},  \ldots  , \bar{I}^{s-1} ] )  \\
 =  (-1)^{\lbrace ( a_{2}-1)  \rbrace}
 \operatorname*{sgn}(\tilde{j}_{q},\tilde{I}^{p})\operatorname*{sgn}  ( \bar{J}_{s-1}; \lbrace \bar{I}^{1}  \ldots   \bar{I}^{s-1}\rbrace )\\
 \operatorname*{sgn}([\tilde{j}_{q}, \tilde{I}^{p}], [( \bar{J}_{s-1};  \lbrace \bar{I}^{1} \ldots   \bar{I}^{s-1} \rbrace ) ]),
\end{multline*}
 and hence, in a manner analogous to the proof of \eqref{formula sign adjugate k even}, we have,
  \begin{multline}\label{formula sign adjugate k odd}
\frac{\left(\operatorname*{adj}\nolimits_{s} X \right)^{ [ I^{m}, \bar{I}^{1},  \ldots  , \bar{I}^{s-1} ] }_{ [ j_{l}\bar{J}_{s-1} ]  }}
 {\operatorname*{sgn}([ j_{l}\bar{J}_{s-1} ] ;[ I^{m}, \bar{I}^{1},\ldots  , \bar{I}^{s-1} ] ) } \\
  = - (-1)^{(b_{1}-b_{2})}
  \frac{\left(\operatorname*{adj}\nolimits_{s} X \right)^{  [ \tilde{I}^{p}, \bar{I}^{1},  \ldots  , \bar{I}^{s-1} ] }_{
  [ \tilde{j}_{q} \bar{J}_{s-1} ]  }}
 {\operatorname*{sgn}( [ \tilde{j}_{q} \bar{J}_{s-1} ] ; [ \tilde{I}^{p}, \bar{I}^{1}, \ldots  , \bar{I}^{s-1} ])} ,
\end{multline}
for any $1 \leq s \leq s_{0}+1$, any subset
    $ \bar{J}_{s-1} = \left\lbrace \bar{j}_{1},   \ldots , \bar{j}_{s-1} \right\rbrace
   \subset \breve{J} $ of $s-1$ indices and
any choice of  of $s$ multiindices $ \bar{I}^{1},  \ldots , \bar{I}^{s-1} $ out of $s_{0}+1$ multiindices, where $a_1$ is the position of $j_l$ in $\left[ j_{l}\bar{J}_{s-1}\right]$ ,
$a_2$ is the position of $ \tilde{j}_{q}$ in $\left[ \tilde{j}_{q} \bar{J}_{s-1} \right]$, $b_1$ is the position of $ I^{m}$ in
$[ I^{m}, \bar{I}^{1},  \ldots  , \bar{I}^{s-1} ]$ and
$b_2 $ is the position of $\tilde{I}^{p}$ in $\left[  \tilde{I}^{p}, \bar{I}^{1},  \ldots  , \bar{I}^{s-1} \right]$. Claim \ref{interchange of indices k odd} follows from above. \smallskip

      Note that claim \ref{interchange of indices k odd} and claim \ref{common index have zero coefficients k even} together now rule out the possibility
      that an adjugate with non-zero coefficient can have common indices between the blocks of multiindices in the superscript and proves
      $ d_s  = 0 $ for all $s > [\frac{n}{k}]$. Furthermore,  by claim \ref{interchange of indices k odd}, the coefficients of any two adjugates $\left( d_{s} \right)^{I}_{J} ,
      \left( d_{s} \right)^{\tilde{I}}_{\tilde{J}} $ such that $ I\cup J =\tilde{I}\cup\tilde{J}$, can differ only by a sign.
      So clearly, all of them must be $0$ if one of them is.
      So without loss of generality, we shall restrict our attention to the coefficient of a particularly
      ordered adjugates, one with all distinct indices in subscript and superscripts ,
       for which $ j_1< \ldots <j_s <i_{1}^{1} <\ldots <i_{k-1}^{1}<  \ldots < i_{1}^{s} <  < \ldots < i_{k-1}^{s},$
       henceforth referred to as the totally ordered adjugate,
      Hence for a given $s$, $2 \leq s \leq [\frac{n}{k}],$ and given $\mathcal{I} \in \mathcal{T}^{ks}$, we shall show that,
      \begin{equation}\label{k odd totally ordered zero}
       \left( d_s \right)^{\lbrace i_{1}^{1} i_{2}^{1} \ldots i_{k-1}^{1}\rbrace
       \lbrace i_{1}^{2} i_{2}^{2} \ldots i_{k-1}^{2} \rbrace \ldots \lbrace i_{1}^{s} i_{2}^{s} \ldots i_{k-1}^{s}\rbrace}_{j_1j_2 \ldots j_s} = 0,
      \end{equation}

 where $ j_1< \ldots <j_s < i_{1}^{1} <\ldots <i_{k-1}^{1}<  \ldots < i_{1}^{s}  < \ldots < i_{k-1}^{s}.$
To prove (\ref{k odd totally ordered zero}), we first need the following:
   \begin{claim}\label{k odd interchnage of r indices}
    For any $1 \leq r \leq k-1$, we have,
    \begin{align}\label{interchnage of r indices}
      \left( d_s \right)&^{\lbrace i_{1}^{1} i_{2}^{1} \ldots i_{r}^{1} i_{r+1}^{2} i_{r+2}^{2} \ldots i_{k-1}^{2}\rbrace
       \lbrace i_{r+1}^{1} i_{r+2}^{1} \ldots i_{k-1}^{1} i_{1}^{2} i_{2}^{2} \ldots i_{r}^{2} \rbrace \ldots
       \lbrace i_{1}^{s} i_{2}^{s} \ldots i_{k-1}^{s}\rbrace }_{j_1j_2 \ldots j_s}\notag \\
       & = - \left( d_s \right)^{\lbrace i_{1}^{1} i_{2}^{1} \ldots i_{k-1}^{1}\rbrace
       \lbrace i_{1}^{2} i_{2}^{2} \ldots i_{k-1}^{2} \rbrace \ldots \lbrace i_{1}^{s} i_{2}^{s} \ldots i_{k-1}^{s}\rbrace}_{j_1j_2 \ldots j_s}.
       \end{align}
   \end{claim}

  We prove the claim by induction over $r$. The case for $r=1$ follows from repeated applications
  of claim \ref{interchange of indices k odd} as follows.

Using claim \ref{interchange of indices k odd} to the $k$-flip interchanging $j_{1}$ and $i^{1}_{1}$, then to the $k$-flip interchanging $i^{1}_{1}$ and $i_{1}^{2}$  and  finally to the $k$-flip interchanging $j_{1}$ and $i^{2}_{1}$, we get,
\begin{align*}
    \left( d_s \right)&^{\lbrace i_{1}^{1} i_{2}^{1} \ldots i_{k-1}^{1}\rbrace
       \lbrace i_{1}^{2} i_{2}^{2} \ldots i_{k-1}^{2} \rbrace \ldots \lbrace i_{1}^{s} i_{2}^{s} \ldots i_{k-1}^{s}\rbrace}_{j_1j_2 \ldots j_s}\notag \\
      &=  (-1)^{s} \left( d_s \right)^{\lbrace j_{1} i_{2}^{1} \ldots i_{k-1}^{1}\rbrace
       \lbrace i_{1}^{2} i_{2}^{2} \ldots i_{k-1}^{2} \rbrace \ldots \lbrace i_{1}^{s} i_{2}^{s} \ldots i_{k-1}^{s}\rbrace}_{j_2 \ldots j_s i_{1}^{1}} \\
       &=  - (-1)^{s} \left( d_s \right)^{\lbrace j_{1} i_{2}^{1} \ldots i_{k-1}^{1}\rbrace
       \lbrace i_{1}^{1} i_{2}^{2} \ldots i_{k-1}^{2} \rbrace \ldots \lbrace i_{1}^{s} i_{2}^{s} \ldots i_{k-1}^{s}\rbrace}_{j_2 \ldots j_s i_{1}^{2}} \\
       &= - (-1)^{s}  (-1)^{s-2} \left( d_s \right)^{\lbrace i_{1}^{1} i_{2}^{2} \ldots i_{k-1}^{2}\rbrace
       \lbrace i_{2}^{1} i_{2}^{1} \ldots i_{k-1}^{1} i_{1}^{2}\rbrace \ldots \lbrace i_{1}^{s} i_{2}^{s} \ldots i_{k-1}^{s}\rbrace}_{j_1 j_2 \ldots j_s}  .
\end{align*}

This proves the case for $r =1$.\smallskip

We now assume that (\ref{interchnage of r indices}) is true for $1 \leq r \leq r_{0}-1$ and show the result for $r = r_{0}.$  To show this,
it is enough to prove that for any $2 \leq r_{0} \leq k-1$,
\begin{align}\label{interchnage induction}
      \left( d_s \right)&^{\lbrace i_{1}^{1} i_{2}^{1} \ldots i_{r_{0}-1}^{1} i_{r_{0}}^{2} i_{r_{0}+1}^{2} \ldots i_{k-1}^{2}\rbrace
       \lbrace i_{r_{0}}^{1} i_{r_{0}+1}^{1} \ldots i_{k-1}^{1} i_{1}^{2} i_{2}^{2} \ldots i_{r_{0}-1}^{2} \rbrace \ldots
       \lbrace i_{1}^{s} i_{2}^{s} \ldots i_{k-1}^{s}\rbrace }_{j_1j_2 \ldots j_s}\notag \\
       & = \left( d_s \right)^{\lbrace i_{1}^{1} i_{2}^{1} \ldots i_{r_{0}-1}^{1} i_{r_{0}}^{1} i_{r_{0}+1}^{2} \ldots i_{k-1}^{2}\rbrace
       \lbrace i_{r_{0}+1}^{1} i_{r_{0}+2}^{1} \ldots i_{k-1}^{1} i_{1}^{2} i_{2}^{2} \ldots i_{r_{0}}^{2} \rbrace \ldots
       \lbrace i_{1}^{s} i_{2}^{s} \ldots i_{k-1}^{s}\rbrace }_{j_1j_2 \ldots j_s}.
       \end{align}
Indeed the result for $r = r_{0}$ follows by combining the induction hypothesis and \eqref{interchnage induction}.
The proof is similar to the case for $ r = 1$. Indeed, by applying claim \ref{interchange of indices k odd} to the
$k$-flip interchanging $j_{1}$ and $i^{1}_{r_{0}}$, then to the $k$-flip interchanging $i^{1}_{r_{0}}$ and $i^{2}_{r_{0}}$ and finally
to the $k$-flip interchanging $j_{1}$ and $i^{2}_{r_{0}}$, we deduce ,
\begin{align*}
    \left( d_s \right)&^{\lbrace i_{1}^{1} \ldots i_{r_{0}-1}^{1}i_{r_{0}}^{2} \ldots i_{k-1}^{2}\rbrace
       \lbrace i_{r_{0}}^{1}i_{r_{0}+1}^{1}\ldots i_{k-1}^{1} i_{1}^{2} \ldots i_{r_{0}-1}^{2} \rbrace
       \ldots \lbrace i_{1}^{s} i_{2}^{s} \ldots i_{k-1}^{s}\rbrace}_{j_1j_2 \ldots j_s}\notag \\
      &=  (-1)^{s-1} \left( d_s \right)^{\lbrace j_{1} i_{r_{0}+1}^{1}\ldots i_{k-1}^{1} i_{1}^{2} \ldots i_{r_{0}-1}^{2}\rbrace
       \lbrace i_{1}^{1} \ldots i_{r_{0}-1}^{1}i_{r_{0}}^{2}\ldots i_{k-1}^{2} \rbrace \ldots \lbrace i_{1}^{s} i_{2}^{s} \ldots i_{k-1}^{s}
       \rbrace}_{j_2 \ldots j_s i_{r_{0}}^{1}}\\
       &=  - (-1)^{s-1} \left( d_s \right)^{\lbrace j_{1} i_{r_{0}+1}^{1} \ldots i_{k-1}^{1} i_{1}^{2} \ldots i_{r_{0}-1}^{2}\rbrace
       \lbrace i_{1}^{1} \ldots i_{r_{0}}^{1}  i_{r_{0}+1}^{2} \ldots i_{k-1}^{2} \rbrace
       \ldots \lbrace i_{1}^{s} i_{2}^{s} \ldots i_{k-1}^{s}\rbrace}_{j_2 \ldots j_s i_{r_{0}}^{2}} \\
       &= - (-1)^{s-1}  (-1)^{s-2} \left( d_s \right)^{\lbrace i_{1}^{1} \ldots i_{r_{0}}^{1}  i_{r_{0}+1}^{2} \ldots i_{k-1}^{2} \rbrace
       \lbrace i_{r_{0}+1}^{1} \ldots i_{k-1}^{1} i_{1}^{2} \ldots i_{r_{0}}^{2} \rbrace \ldots
       \lbrace i_{1}^{s} i_{2}^{s} \ldots i_{k-1}^{s}\rbrace}_{j_1 j_2 \ldots j_s}  .
\end{align*}
This proves \eqref{interchnage induction}) and establishes claim \ref{k odd interchnage of r indices}.\smallskip

Now, using claim \ref{k odd interchnage of r indices}, in particular for $r = k-1$, we obtain,

\begin{align*}
      \left( d_s \right)&^{\lbrace i_{1}^{1} i_{2}^{1} \ldots i_{k-1}^{1}\rbrace
       \lbrace i_{1}^{2} i_{2}^{2} \ldots i_{k-1}^{2} \rbrace \ldots \lbrace i_{1}^{s} i_{2}^{s} \ldots i_{k-1}^{s}\rbrace}_{j_1j_2 \ldots j_s}\notag \\
       & = - \left( d_s \right)^{\lbrace i_{1}^{1} i_{2}^{1} \ldots i_{k-1}^{1}\rbrace
       \lbrace i_{1}^{2} i_{2}^{2} \ldots i_{k-1}^{2} \rbrace \ldots \lbrace i_{1}^{s} i_{2}^{s} \ldots i_{k-1}^{s}\rbrace}_{j_1j_2 \ldots j_s}.
       \end{align*}
This proves \eqref{k odd totally ordered zero} and finishes the proof of the lemma in the case  when $k$ is odd and
thereby establishes lemma \ref{polyconvexitylemma} in all cases. \end{proof}

\section{Proof of the main theorem}
We start by recalling a result regarding ext. polyconvex functions which we will use later. See \cite{BDS1} (cf. Proposition 14(ii)) for the proof.
\begin{proposition}
\label{Proposition equiv polyconvexite}Let $f:\Lambda^{k}  \rightarrow\mathbb{R}.$ Then $f$ is ext. polyconvex if and only if, for every $\xi
\in\Lambda^{k}$ and $1\leq s\leq\left[  n/k\right]  ,$ there exists $c_{s}=c_{s}\left(  \xi\right)  \in\Lambda^{ks} $
 such that%
\[
f\left(  \eta\right)  \geq f\left(  \xi\right)  +\sum_{s=1}^{\left[
n/k\right]  }\left\langle c_{s}\left(  \xi\right)  ;\eta^{s}-\xi
^{s}\right\rangle ,\quad\text{for every }\eta\in\Lambda^{k}.
\]
\end{proposition}

Now we are ready to prove the main theorem.

\begin{proof}[Proof of Theorem \ref{Prop ext quasi implique quasi}]
\textbf{(i)} Recall (cf. Proposition \ref{Prop. proj rang un et quasi}) that%
\[
\pi\left(  \alpha\otimes\beta\right)  =\alpha\wedge\beta.
\]
The rank one convexity of $f\circ\pi$ follows then at once from the ext. one
convexity of $f.$ We now prove the converse. Let $\xi\in\Lambda^{k},$
$\alpha\in\Lambda^{k-1}$ and $\beta\in\Lambda^{1};$ we have to show that%
\[
g:t\rightarrow g\left(  t\right)  =f\left(  \xi+t\,\alpha\wedge\beta\right)
\]
is convex. Since the map $\pi$ is onto, we can find $\Xi\in\mathbb{R}%
^{\tbinom{n}{k-1}\times n}$ so that $\pi\left(  \Xi\right)  =\xi.$ Therefore,
\begin{align*}
 g\left(  t\right) =f\left(\pi\left(  \Xi\right)  +t\,\pi\left(  \alpha\otimes\beta\right)  \right)
=f\left(  \pi\left(  \Xi+t\,\alpha\otimes\beta\right)  \right) ,
\end{align*}

and the convexity of $g$ follows at once from the rank one convexity of
$f\circ\pi.$\smallskip

\textbf{(ii)} Similarly since (cf. Proposition
\ref{Prop. proj rang un et quasi}) $\pi\left(  \nabla\omega\right)  =d\omega,$
we immediately infer the quasiconvexity of $f\circ\pi$ from the ext.
quasiconvexity of $f.$ The reverse implication follows also in the same manner.\smallskip

\textbf{(iii)}
Since $f$ is ext. polyconvex 
we can find, using proposition \ref{Proposition equiv polyconvexite}, for every $\alpha
\in\Lambda^{k} $ and $1\leq s\leq [\frac{n}{k}],$ $c_{s}=c_{s}\left(  \alpha\right)  \in\Lambda^{ks},$
 such that%
\[
f\left(  \beta\right)  \geq f\left(  \alpha\right)  +\sum_{s=1}^{\left[
n/k\right]  }\left\langle c_{s}\left(  \alpha\right)  ;\beta^{s}-\alpha
^{s}\right\rangle ,\quad\text{for every }\beta\in\Lambda^{k}.
\]
Appealing to the proposition \ref{formula} we get, for every $\xi \in \mathbb{R}^{\tbinom{n}{k-1}\times n}$,%
\begin{align*}
f\left(  \pi\left(  \eta\right)  \right)
&  \geq f\left(  \pi\left(
\xi\right)  \right)  +\sum_{s=1}^{\left[  n/k\right]  }\left\langle
c_{s}\left(  \pi\left(  \xi\right)  \right)  ;\left[  \pi\left(  \eta\right)
\right]  ^{s}-\left[  \pi\left(  \xi\right)  \right]  ^{s}\right\rangle
\smallskip\\
&  =f\left(  \pi\left(  \xi\right)  \right)  +\sum_{s=1}^{\left[  n/k\right]
}\left\langle \widetilde{c}_{s}\left(  \xi\right)  ;\operatorname*{adj}%
\nolimits_{s}\eta-\operatorname*{adj}\nolimits_{s}\xi\right\rangle,
\end{align*}
for every $\eta \in \mathbb{R}^{\tbinom{n}{k-1}\times n}$,
which shows that $f\circ\pi$ is indeed polyconvex  by theorem 5.6 in \cite{DCV2} .\smallskip

We now prove the reverse implication. Take $N = \tbinom{n}{k-1} $.
Since $f\circ\pi$ is polyconvex, we have, using theorem 5.6 in \cite{DCV2} again, for every $\xi \in \mathbb{R}^{N\times n}$,
there exists $d_s=d_{s}\left(  \xi\right)  \in \mathbb{R}^{\tbinom{N}{s} \times \tbinom{n}{s}}$ for
  all $1 \leq s \leq \min\left\{  N, n  \right\}$ such that
\begin{equation}\label{fpipolyconvex}
 f\left(  \pi\left(  \eta\right)  \right)     \geq f\left(  \pi\left(
\xi\right)  \right)  + \sum_{s=1}^{\min\left\{  N, n  \right\}}   \left\langle
d_{s}\left(  \xi\right)  ;\operatorname*{adj}\nolimits_s \eta - \operatorname*{adj}\nolimits_s \xi \right\rangle,
\end{equation}
for every $\eta \in \mathbb{R}^{N\times n}$.\smallskip

But this means that there exists $d$, given by $d = (  d_1,\ldots, d_{\min\left\{  N, n \right\}} )$ such that
the function $X \mapsto g(X,d)$, where $g(X,d)$ is as defined in lemma \ref{polyconvexitylemma}, achieves a minima at $X=\xi$.
Then lemma \ref{polyconvexitylemma} implies, for every $1 \leq s   \leq \min\left\{  N, n  \right\}$, there exists $\mathcal{D}_{s} \in \Lambda^{ks}$
such that
\begin{equation*}
\left\langle d_s, \operatorname*{adj}\nolimits_s \eta - \operatorname*{adj}\nolimits_s \xi \right\rangle = \left\langle
\mathcal{D}_{s}  ;\pi_s(\operatorname*{adj}\nolimits_s \eta )- \pi_s(\operatorname*{adj}\nolimits_s \xi ) \right\rangle ,
\end{equation*}
 for every $ \eta \in \mathbb{R}^{N\times n}.$
Hence, we obtain from \eqref{fpipolyconvex}, for every $\xi \in \mathbb{R}^{ N \times n}$,
\begin{align} \label{fpipolyconvexityintermediate}
 f\left(  \pi\left(  \eta\right)  \right)     \geq &\qquad f\left(  \pi\left(
\xi\right)  \right) + \sum_{s=1}^{\left[ n/k\right]}   \left\langle
\mathcal{D}_{s} \left(  \xi\right)  ;\pi_s(\operatorname*{adj}\nolimits_s \eta )- \pi_s(\operatorname*{adj}\nolimits_s \xi ) \right\rangle ,
\end{align}
for every  $\eta \in \mathbb{R}^{ N \times n}$.
Since $\pi$ is onto, given any $\alpha, \beta \in\Lambda^{k}$, we can find $\eta, \xi \in \mathbb{R}^{ N \times n}$
such that $\pi(\eta)=\beta$ and $\pi(\xi) = \alpha$. Now using \eqref{fpipolyconvexityintermediate} and the definition of $\pi_s$, we have,
by defining $c_s (\alpha) = \mathcal{D}_{s} (\xi) $,
for every $\alpha
\in\Lambda^{k}$,
\[
f\left(  \beta\right)  \geq f\left(  \alpha\right)  +\sum_{s=1}^{\left[
n/k\right]  }\left\langle c_{s}\left(  \alpha\right)  ;\beta^{s}-\alpha
^{s}\right\rangle ,\quad\text{for every }\beta\in\Lambda^{k}.
\]

This proves $f$ is ext. polyconvex 
and concludes the proof of the theorem.
\end{proof} \smallskip
\section{Notations}\label{notations}
We gather here the notations which we will use throughout this article.
\begin{enumerate}
\item Let $k$ be a nonnegative integer and $n$ be a positive integer.
\begin{itemize}
\item We write $\Lambda^{k}\left(  \mathbb{R}^{n}\right)  $ (or simply
$\Lambda^{k}$) to denote the vector space of all alternating $k-$linear maps
$f:\underbrace{\mathbb{R}^{n}\times\cdots\times\mathbb{R}^{n}}_{k-\text{times}%
}\rightarrow\mathbb{R}.$ For $k=0,$ we set $\Lambda^{0}\left(  \mathbb{R}%
^{n}\right)  =\mathbb{R}.$ Note that $\Lambda^{k}\left(  \mathbb{R}%
^{n}\right)  =\{0\}$ for $k>n$ and, for $k\leq n,$ $\operatorname{dim}\left(
\Lambda^{k}\left(  \mathbb{R}^{n}\right)  \right)  ={\binom{{n}}{{k}}}.$

\item $\wedge,$ $\lrcorner\,,$ $\left\langle\  ;\ \right\rangle $ and $\ast$ denote the exterior product, the interior product, the
scalar product and the Hodge star operator respectively.

\item If $\left\{  e^{1},\cdots,e^{n}\right\}  $ is a basis of $\mathbb{R}%
^{n},$ then, identifying $\Lambda^{1}$ with $\mathbb{R}^{n},$%
\[
\left\{  e^{i_{1}}\wedge\cdots\wedge e^{i_{k}}:1\leq i_{1}<\cdots<i_{k}\leq
n\right\}
\]
is a basis of $\Lambda^{k}.$ An element $\xi\in\Lambda^{k}\left(
\mathbb{R}^{n}\right)  $ will therefore be written as%
\[
\xi=\sum_{1\leq i_{1}<\cdots<i_{k}\leq n}\xi_{i_{1}i_{2}\cdots i_{k}%
}\,e^{i_{1}}\wedge\cdots\wedge e^{i_{k}}=\sum_{I\in\mathcal{T}^{k}}\xi
_{I}\,e^{I}%
\]
where%
\[
\mathcal{T}^{k}=\left\{  I=\left(  i_{1}\,,\cdots,i_{k}\right)
\in\mathbb{N}^{k}:1\leq i_{1}<\cdots<i_{k}\leq n\right\}  .
\]
An element of $\mathcal{T}^{k}$ will be referred to as a multiindex.
We adopt the alphabetical order for comparing two multiindices and we do not reserve a specific symbol for this ordering. The usual ordering symbols, when
 written in the context of multiindices will denote alphabetical ordering.

\item We write%
\[
e^{i_{1}}\wedge\cdots\wedge\widehat{e^{i_{s}}}\wedge\cdots\wedge e^{i_{k}%
}=e^{i_{1}}\wedge\cdots\wedge e^{i_{s-1}}\wedge e^{i_{s+1}}\wedge\cdots\wedge
e^{i_{k}}.
\]
 Similarly, $ \widehat{\hphantom{e^{i_{s}}} } $ placed over a string of indices (or multiindices )  will signify the omission
 of the string under the $ \widehat{\hphantom{e^{i_{s}}} } $ sign.
\end{itemize}
\item Let $\Omega\subset\mathbb{R}^{n}$ be a bounded open set.

\begin{itemize}
\item The spaces $C^{1}\left(  \Omega;\Lambda^{k}\right)  ,$ $W^{1,p}\left(
\Omega;\Lambda^{k}\right)  $ and $W_{0}^{1,p}\left(  \Omega;\Lambda
^{k}\right)  ,$ $1\leq p\leq\infty$ are defined in the usual way.

\item For any $\omega\in W^{1,p}\left(  \Omega;\Lambda^{k}\right)  ,$ the exterior
derivative $d\omega$ belongs to $L^{p}(\Omega;\Lambda^{k+1})$ and is defined
by, for all $1\leq i_{1}<\cdots<i_{k+1}\leq n$,
\[
(d\omega)_{i_{1}\cdots i_{k+1}}=\sum_{j=1}^{k+1}\left(  -1\right)  ^{j+1}%
\frac{\partial\omega_{i_{1}\cdots i_{j-1}i_{j+1}\cdots i_{k+1}}}{\partial
x_{i_{j}}}\,,
\]
\end{itemize}
\item \emph{Notation for indices:} The following system of notations will be employed throughout.
\begin{itemize}

 \item[(i)] Single indices will be written as lower case english letters, multiindices will be written as upper case english letters.

 \item[(ii)] Multiindices will always be indexed by superscripts. The use of a subscript while writing a multiindex is reserved for a special purpose.
See (vi) below.

 \item[(iii)]  $\lbrace i_1 \ldots  i_r \rbrace$  will represent
  the string of indices  $i_1  \ldots  i_r$. In the same way,  $\lbrace I^1  \ldots  I^r \rbrace$
  will represent the string of multiindices obtained by writing out the multiindices in the indicated order.

  \item[(iv)] $\left( i_1  \ldots  i_r \right)$ will stand for the permutation of the $r$ indices that arranges the string
  $\lbrace i_1  \ldots  i_r \rbrace$ of distinct indices
   in strictly increasing order.

  \item[(v)]  $[i_1  \ldots i_r]$   will stand for the increasingly ordered string of indices consisting of the distinct
indices $i_1,  \ldots , i_r$. However,  $[I^1,  \ldots, I^r]$ will represent the corresponding string of distinct multiindices
$I^1,  \ldots, I^r$, arranged in the increasing alphabetical order, whereas $[I^1  \ldots I^r]$ will represent the string of indices obtained by
arranging all the  distinct single indices contained in the multiindices $I^1,  \ldots, I^r$ in increasing order.


  \item[(vi)] For $I \in \mathcal{T}^{k}$ and $j \in I$, $I_{j}$  stands for the multiindex obtained by removing $j$ from $I$.

  \item[(vii)] The symbol  $\left( J; I \right) $, where $J = \lbrace j_1  \ldots j_s\rbrace$ is a string of $s$ single indices
  and $I = \lbrace I^1  \ldots I^s\rbrace$ is a string of $s$ multiindices, $I^1, \ldots , I^s \in \mathcal{T}^{(k-1)s}$,
  will be reserved to denote the interlaced string
  $ \left\lbrace j_{1} I^{1} \ldots j_{s} I^{s} \right\rbrace $.
 \item[(viii)] The abovementioned system of notations will be in force even when representing indices as subscripts of superscripts of different objects.
\end{itemize}
\item \emph{Flip:} We shall be employing some particular permutations often.
\begin{definition}[$1$-flip]
Let $s \geq 1$, let $J \in \mathcal{T}^{s}$, $I \in \mathcal{T}^{l}$  be written as,
$J = \lbrace j_{1}  \ldots  j_{s} \rbrace $, $I =  \lbrace i_{1} \ldots i_{l} \rbrace$  with $J \cap I = \emptyset$.
Let $\tilde{J} \in \mathcal{T}^{s}$, $\tilde{I} \in \mathcal{T}^{l}$. We say that $( \tilde{J}, \tilde{I} )$ is obtained from $( J, I )$
by a $1$-flip interchanging $j_{p}$ with $i_{m}$, for some $1 \leq p \leq s$, $1 \leq m \leq l$, if
 $$ \tilde{J} = \left[ j_{1}\ldots j_{p-1} i_{m} i_{p+1} \ldots j_{l}  \right] \text{ and }
  \tilde{I} = \left[  i_{1}\ldots i_{m-1} j_{p} i_{m+1} \ldots i_{k} \right] .$$

\end{definition}
\begin{definition}[$k$-flip]\label{kflips}
 Let $s \geq 1$, $k \geq 2$. Let $J \in \mathcal{T}^{s},$ $ J = \lbrace j_{1} \ldots  j_{s}\rbrace $,
 $I = \lbrace I^{1}  \ldots I^{s}\rbrace
 = [ I^{1},  \ldots, I^{s}]$,  where $ I^{1},  \ldots, I^{s}
  \in \mathcal{T}^{k}$, $I^{r} =  \lbrace  i^{r}_{1},  \ldots , i^{r}_{k}\rbrace$ for all
  $1 \leq r \leq s$ and $J \cap I =  \emptyset$. We
 say that $(\tilde{J}, \tilde{I} )$ is obtained from $(J, I)$ by a $k$-flip if
 there exist integers $1 \leq m, p \leq s $ and $1 \leq q \leq k $ such that,
 $$\tilde{J}= [j_{1}  \ldots j_{p-1} i^{m}_{q} j_{p+1} \ldots  j_{s}],$$  and $$\tilde{I}=[ I^{1},
 \ldots I^{m-1}, [i^{r}_{1}  \ldots i^{r}_{q-1} j_{p} i^{r}_{q+1} \ldots i^{r}_{k} ], I^{m+1}, \ldots , I^{s}  ] .$$
\end{definition}
Note that a $k$-flip can be seen as a permutation in an obvious way.

\item \emph{Notation for sum:}
For $I \in \mathcal{T}^{ks}$, where $1 \leq k \leq n $ and  $1 \leq s \leq [\frac{n}{k}]$, the shorthand $\sum\nolimits^I_s $ stands for the sum,    $$ \sum_{\substack{J, \tilde{I} \\ J = \lbrace j_1\ldots j_s\rbrace = [j_1\ldots j_s],\\
\tilde{I} = \lbrace I^1\ldots I^s \rbrace = [I^1,\ldots , I^s]\\
J\cup\tilde{I} = I}}.$$
\end{enumerate}
\textbf{Acknowledgement.} We have benefitted of interesting discussions with Professor Bernard Dacorogna. Part of this work was completed during visits of S. Bandyopadhyay to EPFL, whose hospitality and support is gratefully acknowledged.
The research of S. Bandyopadhyay was partially supported by a SERB research project titled  ``Pullback Equation for Differential Forms".

\bibliographystyle{plain}

\end{document}